\documentclass[10pt,draft,reqno]{amsart}
     \makeatletter
     \def\section{\@startsection{section}{1}%
     \z@{.7\linespacing\@plus\linespacing}{.5\linespacing}%
     {\bfseries
     \centering
     }}
     \def\@secnumfont{\bfseries}
     \makeatother
\newtheorem{theorem}{Theorem}[section]
\newtheorem{lemma}[theorem]{Lemma}
\newtheorem{proposition}[theorem]{Proposition}
\newtheorem{corollary}[theorem]{Corollary}
\theoremstyle{definition}

\theoremstyle{remark}
\newtheorem{remark}[theorem]{Remark}
\numberwithin{equation}{section}
\newcommand{\Gm}{\Gamma}
\newcommand{\dl}{\delta}

\newcommand{\eqd}{\overset{\mathrm d}{=}}
\newcommand{\eqc}{\overset{\mathrm d}{\Rightarrow}}

\newcommand{\R}{\mathbb{R}}
\newcommand{\N}{\mathbb{N}}

\newcommand{\sek}{\int_0^{\infty}}

\newcommand{\n}{\noindent}

\begin{document}

\title[Bifractional noises]{Limits of bifractional Brownian noises}

\author{Makoto Maejima and Ciprian A. Tudor}
\address{Makoto Maejima: Department of Mathematics, Keio University, 3-14-1, 
Hiyoshi, Kohoku-ku, Yokohama 223-8522,
Japan} \email{maejima@math.keio.ac.jp}

\address{Ciprian A. Tudor:  SAMOS-MATISSE, Centre d'Economie de La
Sorbonne,  Universit\'{e} de Panth\'eon-Sorbonne Paris 1, 90, rue de
Tolbiac, 75634 Paris Cedex 13, France.} \email{tudor@univ-paris1.fr}



\begin{abstract}
Let $B^{H,K}=\left (B^{H,K}_{t}, t\geq 0\right )$ be a bifractional Brownian motion 
with two parameters $H\in (0,1)$ and $K\in(0,1]$.
The main result of this paper is that the increment process
generated by the bifractional Brownian motion $\left( B^{H,K}_{h+t}
-B^{H,K} _{h}, t\geq 0\right)$ converges when $h\to \infty$ to
$\left (2^{(1-K)/{2}}B^{HK} _{t}, t\geq 0\right )$, where $\left (B^{HK}_{t}, t\geq 0\right)$
is the fractional Brownian motion with Hurst index $HK$. We also
study the behavior of the noise associated to the bifractional
Brownian motion and limit theorems to $B^{H,K}$.
\end{abstract}

\maketitle

\allowdisplaybreaks
\noindent

\section{Introduction}

Introduced in \cite{HV03}, the bifractional Brownian motion, a
generalization of the fractional Brownian motion,  has been studied
in many aspects (see \cite{AL}, \cite{ET07}, \cite{KRT}, \cite{LN08}, \cite{NO2}, \cite{RT06} and
\cite{TX07}). This stochastic process
is defined as follows. Let $H\in (0,1)$ and $K\in (0,1]$. Then
$B^{H,K}= \left (B_t^{H,K}, t\ge 0\right)$ is a centered Gaussian process with
covariance
$$
E\left [B^{H,K}_{t}B^{H,K}_{s}\right ] = 2^{-K}\left (
(t^{2H}+s^{2H})^K- |t-s|^{2HK}\right).
$$
When $K=1$, it is the fractional Brownian motion $B^H=\left (B_t^H, t\ge0\right )$ 
with the Hurst index $H\in (0,1)$.  
In general, the process
$B^{H,K}$ has the following basic properties: It is a selfsimilar
stochastic process of order $HK\in (0,1)$, the increments are not
stationary and it is a quasi-helix  in the sense of \cite{Kahane}
since for every $s,t\geq 0$, we have
\begin{equation*}
\label{ineq} 2^{-K}\vert t-s \vert ^{2HK}\leq E\left[\left(
B_{t}^{H,K}-B_{s}^{H,K}\right )^2\right] \leq 2^{1-K}\vert t-s \vert
^{2HK}.
\end{equation*}
The trajectories of the process $B^{H,K}$ are $\delta$-H\"older
continuous for any $\delta <HK$ and they are nowhere differentiable.

A better understanding of this process has been presented in the
paper \cite{LN08}, where the authors showed a decomposition of
$B^{H,K}$ with $H,K\in (0,1)$  as follows. Let
$(W_{\theta}, \theta\ge 0)$ be a standard Brownian motion
independent of $B^{H,K}$. For any $K\in (0,1)$, they defined a
centered Gaussian process $X^K=\left (X_t^K, t\ge 0\right)$ by
\begin{equation}
X_t^K = \sek (1-e^{-\theta t})\theta ^{-(1+K)/2}dW_{\theta}.
\end{equation}
Its covariance is
\begin{equation}
\label{eq:x-cov} E\left [X_t^KX_s^K\right] = \Gm(1-K)K^{-1}\left
(t^K+s^K-(t+s)^K\right).
\end{equation}
Then they showed, by setting
\begin{equation}
\label{eq:x-hk} X_t^{H,K}:=X^K_{t^{2H}},
\end{equation}
that \begin{equation} \label{deco}
 \left (C_1X_t^{H,K}+B_t^{H,K}, t\ge 0\right ) \eqd \left ( C_2B_t^{HK}, t\ge 0\right ),
\end{equation}
where $C_1= (2^{-K}K(\Gm(1-K))^{-1})^{1/2}, C_2= 2^{(1-K)/2}$ and
$\eqd$ means equality of all finite dimensional distributions.

The main purpose of this paper is to study the increment process
$$\left( B^{H,K}_{h+t} -B^{H,K} _{h}, t\geq 0\right)$$ (where $h\geq
0$) of $B^{H,K}$ and the noise generated by
$B^{H,K}$ and to see how close this process is to a process with
stationary increments. In principle, since the bifractional Brownian
motion is not a process with stationary increments, its increment
process depends on $h$. 
But in this paper we show, by using the
decomposition (\ref{deco}), that for $h\to \infty$ the increment
process  $\left( B^{H,K}_{t+h} -B^{H,K} _{h}, t\geq 0\right)$
converges,  modulo a constant,  to the fractional Brownian motion with Hurst index
$HK$  in the sense of finite dimensional distributions, 
so the dependence of the increment process depending on $h$ decreases
for very large $h$. Somehow, one can interpret that,
for very big starting point, the bifractional Brownian motion has
stationary increments. Then we will try to understand this property
from the perspective of the \lq\lq noise" generated by $B^{H,K}$
i.e. the Gaussian sequence $B^{H,K}_{n+1} -B^{H,K}_{n}$, where
$n\geq 0$ are integers. The behavior of the sequence
$$
Y_{a}(n)=    E\left[ \left (B^{H,K}_{a+1} -B^{H,K} _{a}\right )\left
(B^{H,K}_{a+n+1} -B^{H,K}_{a+n}\right )\right],\quad a\in\mathbb N,
$$
(which, if $K=1$,  is constant with respect to $a$ and of order
$n^{2H-2}$) is studied with respect to $a$ and with respect to $n$
in order to understand the contributions of $B^{HK}$ and $X^{H,K}$.

We organize our paper as follows. In Section 2 we prove our
principal result which says that the increment process of $B^{H,K}$
converges to the fractional Brownian motion $B^{HK}$. Sections 3-5
contain some consequences and different views of this main result.
We analyze the noise generated by the bifractional Brownian motion
and we study its asymptotic behavior and we interpret the process
$X^{H,K}$ as the difference between "the even part" and "the odd
part" of the fractional Brownian motion. Finally, in Section 6 we
prove  limit theorems to the bifractional Brownian from a correlated
non-stationary Gaussian sequence.

\vskip 10mm

\section{The limiting process of the bifractional Brownian motion}

In this section, we prove the following main result; it says that
the increment process of the bifractional Brownian motion converges
to the fractional Brownian motion with Hurst index $HK$.
\begin{theorem}Let $K\in (0,1)$. Then
\label{thm:limit}
$$
\left ( B_{h+t}^{H,K} - B_h^{H,K}, t\ge 0\right ) \eqc \left
(2^{(1-K)/2}B_t^{HK} ,t\ge 0\right) \quad\text{as}\,\, h\to \infty,
$$
where $\eqc$ means convergence of all finite dimensional
distributions.
\end{theorem}
To prove Theorem~\ref{thm:limit}, we use the decomposition
(\ref{deco}). It is enough to show that the increment process
associated to $X^{H,K}$ converges to zero; we prove it in the next
result, and actually we measure how fast it tends to zero with
respect to $L^{2}$ norm. It will be useful to compare this rate of
convergence with results in the following sections.

\begin{proposition}\label{Xh}Let $X^{H,K}$ be the process given by {\rm (\ref{eq:x-hk})}. Then as $h\to \infty$
\begin{equation*}
E\left [\left(  X_{h+t}^{H,K} - X_h^{H,K} \right )^2\right ]=\Gamma
(1-K)K^{-1} 2^{K}H^{2}K (1-K) t^{2}{h^{2(HK-1)} }(1+o(1)).
\end{equation*}
As a consequence,
$$
\left ( X_{h+t}^{H,K} - X_h^{H,K}, t\ge 0\right ) \eqc (X(t)\equiv
0, t\ge 0)\quad\text{as} \,\, h\to \infty.
$$
\end{proposition}
\begin{proof} Note from \eqref{eq:x-cov} and \eqref{eq:x-hk} that
$$
E\left [ X_t^{H,K}X_s^{H,K}\right] =\Gm(1-K)K^{-1}\left(
t^{2HK}+s^{2HK}-\left(t^{2H}+s^{2H}\right)^K\right )
$$
and in particular, for every $t\geq 0$
$$  
E \left[ \left ( X_{t} ^{H,K} \right ) ^{2}\right] = \Gamma (1-K) K^{-1} (2-2^{K} )
t^{2HK}.
$$ 
We have
\begin{align*}
 E & \left [ \left ( X_{h+t}^{H,K} - X_h^{H,K}\right )^2\right ]
= E\left [\left ( X_{h+t}^{H,K}\right )^2\right] -2 E\left [
X_{h+t}^{H,K}X_h^{H,K}\right]+ E\left [\left ( X_{h}^{H,K}\right
)^2\right].
\end{align*}
Then
\begin{align*}
I &:=K (\Gamma (1-K)) ^{-1}E  \left [ \left(  X_{h+t}^{H,K} -
X_h^{H,K}\right ) ^2\right ]\\
&= \Big( (2-2^{K} ) (h+t) ^{2HK} \\
&\hskip 10mm -2 \left( (h+t) ^{2HK} + h^{2HK}
-\left( (h+t)^{2H} + h^{2H} \right) ^{K} \right)
+ (2-2^{K}) h^{2HK}\Big)\\
&=-2^{K} \left( (h+t) ^{2HK} + h^{2HK} \right) + 2 \left( (h+t)
^{2H} + h^{2H} \right)^{K}\\
&= -2^{K} h^{2HK} \left(1+ (th^{-1} ) ^{2HK} +1 \right)+
2h^{2HK}\left(  (1+th^{-1} ) ^{2H} + 1 \right) ^{K}.
\end{align*}
Therefore for very large $h>0$ we obtain by using Taylor's expansion
\begin{align*}
I= -2^{K} &h^{2HK} \left(   2 + 2HK th^{-1} + H(2H-1) t^{2} h^{-2}
(1+ o(1)) \right)\\
&  + 2h^{2HK} \left(2 + 2Hth^{-1} + H(2H-1) t^{2}h^{-2} (1+ o(1))
\right)^{K}.
\end{align*}
Now we use Taylor expansion for the function $(2+Z) ^{K}$ for $Z$
close to zero. In our case  $Z=2Hth^{-1} +H(2H-1) t^{2} h^{-2}+
r(h)$ with $r(h) h^{2} \to 0$ as $h\to \infty$. We obtain
\begin{align*}
I&=-2^{K} h^{2HK} \left ( 2 + 2HK th^{-1} + H(2H-1) t^{2} h^{-2}
(1+ o(1)) \right)\\
& \hskip 10mm + 2h^{2HK} \Big( 2^{K} + K2^{K-1} (2Hth^{-1} +H(2H-1) t^{2}
h^{-2} +r(h) ) \\
& \hskip 15mm + 2^{-1}{K(K-1)} 2^{K-2}(2Hth^{-1} +H(2H-1) t^{2}
h^{-2} +r(h) ) ^{2} + o(h^{-2} ) \Big)\\
&= h^{2HK}2^{K} HK ( -2HK +1 +2H -1+ HK-H)t^{2}h^{-2} (1+o(1))\\
& = h^{2HK} 2^{K} H^{2} K(1-K)t^{2}h^{-2} (1+o(1)).
\end{align*}
Consequently, we have
\begin{align*}
 E  \left [ \left(  X_{h+t}^{H,K} -
X_h^{H,K}\right ) ^2\right ]= \Gamma (1-K)K^{-1}
2^{K}H^2K(1-K)t^2h^{2(HK-1)}(1+o(1)),
\end{align*}
which tends to $0$ as $h\to\infty$, since $HK-1 <0$. \end{proof}

\vskip 10mm

\section{Bifractional Brownian noise}

By considering the bifractional Brownian noise, which are increments
of bifractional Brownian motion, we can understand
Theorem~\ref{thm:limit} in a different way. \n Define for every
integer  $n\geq 0$, the bifractional Brownian noise
$$
Y_{n} =B^{H,K} _{n+1}-B^{H,K} _{n}.
$$

\vskip 3mm
\begin{remark}
Recall that in the fractional Brownian motion case ($K=1$) we have for
every $a\in\N$ and for every $n\ge 0$,
$E\left[ Y_{a} Y_{a+n} \right]  = E\left[ Y_{0} Y_{n}\right].$
\end{remark}
Let us denote 
$$
R(0,n)= E[Y_{0}Y_{n}] = E\left[ B^{H,K}_{1}
\left(B^{H,K}_{n+1}-B^{H,K}_{n}\right )\right]
$$
and 
\begin{equation}
\label{rbi}R(a, a+n)=E\left [ Y_aY_{a+n}\right ]=E\left[ \left (B^{H,K}_{a+1} -B^{H,K}_{a}\right )
\left (B^{H,K}_{a+n+1} -B^{H,K}_{a+n}\right )\right].
\end{equation}
Let us compute the term $R(a,a+n)$ and understand how different it is from $R(0,n)$. 
We have for every $n\ge 1$,
\begin{align}
R(a,a+n) &= 2^{-K} \left(  \left( (a+1) ^{2H} + (a+n+1) ^{2H} \right) ^{K} -n^{2HK} \right.\nonumber \\
& \left.\hskip 20mm -\left( (a+1) ^{2H} + (a+n) ^{2H} \right) ^{K} -(n-1)^{2HK} \right. \nonumber \\
& \left. \hskip 20mm -\left( a ^{2H} + (a+n+1) ^{2H} \right) ^{K} -(n+1)^{2HK} \right. \nonumber \\
& \left. \hskip 20mm +\left( a ^{2H} + (a+n) ^{2H} \right) ^{K} -n^{2HK} \right) \nonumber \\
&=: 2^{-K}(f_{a} (n) + g(n)),\label{faga}
\end{align}
where
\begin{align*}
f_{a}(n)&=\left( (a+1) ^{2H} + (a+n+1) ^{2H} \right) ^{K}-\left( (a+1) ^{2H} + (a+n) ^{2H} \right) ^{K}\\
&\hskip 20mm -\left( a ^{2H} + (a+n+1) ^{2H} \right) ^{K} +\left( a
^{2H} + (a+n) ^{2H} \right) ^{K}
\end{align*}
and for every $n\ge 1$,
$$
g(n)= (n+1)^{2HK} +(n-1)^{2HK} -2n^{2HK}.
$$
\begin{remark}
\n (i) The function $g$ is, modulo a constant, the covariance
function  of the fractional Brownian noise with Hurst index $HK$. Indeed, for
$n\ge 1$,
\begin{equation}\label{gn}
g(n) = 2 E\left [ B_1^{HK}(B_{n+1}^{HK}-B_n^{HK})\right ].
\end{equation}
\vskip 2mm \n (ii) $g$ vanishes if $2HK=1$. \vskip 2mm \n (iii)  The
function $f_{a}$ is a \lq\lq new function" specific to the
bifractional Brownian case. (Note that $f_{a}$ vanishes in the case
$K=1$.) It corresponds to the noise generated by $X^{H,K}$. Indeed,
it follows easily from (\ref{deco}) that
\begin{align} 
f_{a}(n) &=
-2^{K}C_{1}^{2} E\left[\left( X^{H,K}_{a+1}-X^{H,K}_{a}\right)
\left( X^{H,K}_{a+n+1}-X^{H,K}_{a+n}\right)\right] \nonumber\\
& \label{cor} = :-2^{K}C_{1}^{2}R^{X^{H,K}}(a,a+n)
\end{align}
for every $a$ and $n\in\mathbb N$.
\end{remark}
We need to analyze the function $f_{a}$ to understand \lq\lq how
far"  the bifractional Brownian noise is from the fractional
Brownian noise. 
In other words, how far is the bifractional
Brownian motion from a process with stationary increments? 
\vskip3mm

\begin{theorem}
\label{thm:fa}For each $n $ it holds that as $a\to\infty$
\begin{equation*}
{f_{a}(n)} = 2H^{2}K  (K-1)a^{2(HK-1)} (1+o(1)).
\end{equation*} Therefore $\displaystyle{\lim_{a\to\infty}f_a(n)=0}$ for each $n$.
\end{theorem}
The bifractional Brownian noise is not stationary. However, the
meaning of the theorem above is that it converges to a stationary
sequence. 
\vskip 3mm
\begin{proof} We have, for  $a\to\infty$,
\begin{align*}
f_a(n)
& =a^{2HK}\Bigl[\bigl\{(1+a^{-1})^{2H}+(1+(n+1)a^{-1})^{2H}\}^K \\
& \hskip 10mm- \bigl\{(1+a^{-1})^{2H}+(1+na^{-1})^{2H}\}^K
 -\bigl\{1+(1+(n+1)a^{-1})^{2H}\}^K \\
& \hskip 10mm+ \bigl\{1+(1+na^{-1})^{2H}\}^K\Bigr]\\
& = a^{2HK} \Bigl[ \bigl\{ 1+2Ha^{-1} + H(2H-1)a^{-2}(1+o(1)) \\
& \hskip 20mm+ 1+2H(n+1)a^{-1} + H(2H-1)(n+1)^2a^{-2}(1+o(1)) \bigr\}^K\\
& \hskip 10mm - \bigl\{ 1+2Ha^{-1} + H(2H-1)a^{-2}(1+o(1)) \\
& \hskip 20mm+ 1+2Hna^{-1} + H(2H-1)n^2a^{-2}(1+o(1)) \bigr\}^K\\
& \hskip 10mm- \bigl\{ 1+  1+2H(n+1)a^{-1} + H(2H-1)(n+1)^2a^{-2}(1+o(1)) \bigr\}^K\\
& \hskip 10mm- \bigl\{ 1+  1+2H(n+1)a^{-1} + H(2H-1)(n+1)^2a^{-2}(1+o(1))\bigr \}^K\Bigr]\\
& = 2a^{2HK} \Bigl[ \bigl\{ 1+H(n+2)a^{-1} \\
& \hskip 20mm + 2^{-1}H(2H-1)(1+(n+1)^2)a^{-2}(1+o(1))  \bigr\}^K\\
& \hskip 10mm - \bigl\{ 1+H(n+1)a^{-1} + 2^{-1}H(2H-1)(1+n^2)a^{-2}(1+o(1))  \bigr\}^K\\
&  \hskip 10mm - \bigl\{ 1+ H(n+1)a^{-1} + 2^{-1}H(2H-1)(n+1)^2a^{-2}(1+o(1)) \bigr\}^K\\
&  \hskip 10mm +\bigl\{ 1+ Hna^{-1} + 2^{-1}H(2H-1)n^2a^{-2}(1+o(1)) \bigr\}^K\Bigr]\\
& = 2a^{2HK} \Bigl[ \bigl\{1+ K(H(n+2)a^{-1}\\
& \hskip 20mm +2^{-1}H(2H-1)(1+(n+1)^2)a^{-2}(1+o(1)))\\
& \hskip 20mm +2^{-1}K(K-1)(H(n+2)a^{-1}(1+o(1)))^2(1+o(1))\bigr\}\\
& \hskip 10mm-\bigl \{1+ K(H(n+1)a^{-1}+ 2^{-1}H(2H-1)(1+n^2)a^{-2}(1+o(1)))\\
& \hskip 20mm +2^{-1}K(K-1)(H(n+1)a^{-1}(1+o(1)))^2(1+o(1))\bigr\}\\
& \hskip 10mm- \bigl\{1+ K(H(n+1)a^{-1}\\
& \hskip 20mm  + 2^{-1}H(2H-1)(1+(n+1)^2)a^{-2}(1+o(1)))\\
& \hskip 20mm +2^{-1}K(K-1)(H(n+1)a^{-1}(1+o(1)))^2(1+o(1))\bigr\}\\
& \hskip 10mm+\bigl\{1+ K(Hna^{-1}+ 2^{-1}H(2H-1)n^2a^{-2}(1+o(1)))\\
& \hskip 20mm +2^{-1} K(K-1)(Hna^{-1}(1+o(1)))^2(1+o(1))\bigr\}\Bigr]\\
& = 2a^{2HK}\Bigl[ (KH(n+2) - KH(n+1) -KH(n+1) + KHn)a^{-1}\\
& \hskip 10mm + 2^{-1}KH(2H-1)(1+(n+1)^2) + 2^{-1}K(K-1)H^2(1+(n+1)^2) \\
& \hskip 10mm -2^{-1}KH(2H-1)(n^2+1) - 2^{-1}K(K-1)H^2(n+1)^2\\
& \hskip 10mm - 2^{-1}KH(2H-1)(n+1)^2 - 2^{-1}K(K-1)H^2(n+1)^2 \\
& \hskip 10mm +2^{-1}KH(2H-1)n^2 + 2^{-1}K(K-1)H^2n^2)\}a^{-2}(1+o(1))\Bigr]\\
& = 2H^2K(K-1)a^{2(HK-1)}(1+o(1)).
\end{align*}
Since $HK-1<0$, the last term tends to 0 when $a$ goes to the
infinity. \end{proof} 
\vskip3mm

\begin{remark}
The fact that the term $f_{a}(n)$ converges to zero as $a\to \infty$
could be seen by Proposition \ref{Xh} since, using H\"older
inequalities,
$$
R^{X^{H,K}}(a, a+n) \leq \left( E\left[\left( X^{H,K}_{a+1}
-X^{H,K}_{a}\right) ^{2}\right]\right) ^{{1}/{2}} \left(
E\left[\left( X^{H,K}_{a+n+1} -X^{H,K}_{a+n}\right)
^{2}\right]\right) ^{{1}/{2}}
$$
and both factors on the right hand side above are of order
$a^{HK-1}$. The result actually confirms that for large $a$,
$X^{H,K}_{a+n+1}-X^{H,K}_{a+n}$ is very close to  $X^{H,K}_{a+1}
-X^{H,K}_{a}$.
\end{remark}
\vskip 10mm


\section{The behavior of increments of the bifractional Brownian motion}

In this section we continue the study of the bifractional Brownian
noise (\ref{rbi}). We are now interested in the behavior with
respect to $n $ (as $n\to \infty$). We know that as $n\to \infty$
the fractional Brownian noise with Hurst index $HK$ behaves  as
$HK(2HK-~1) n^{2(HK-1)}$. Given the decomposition (\ref{deco}) it is
natural to ask what the contribution of the bifractional Brownian
noise to this is and what the contribution of the process $X^{H,K}$
is. We have the following.
\begin{theorem} For integers $a, n\ge 0$, let $R(a, a+n)$ be given by {\rm (\ref{rbi})}.
Then for large $n$, \label{thm:BHK}
\begin{align*}
R(a, a+n)& =
 2^{-K}\left ( 2HK(2HK-1)n^{2(HK-1)} \right .\\
&\left .\hskip 10mm + HK(K-1)\left( (a+1) ^{2H} -a^{2H} \right)
n^{2(HK-1)+(1-2H)} + \cdots\right ) .
\end{align*}
\end{theorem}

\begin{proof} Recall first that by (\ref{faga}) and (\ref{gn}),
\begin{eqnarray*}
R(a, a+n)= 2^{-K} (f_{a}(n) + g(n))
\end{eqnarray*}
and the term $g(n)$ behaves as $2HK(2HK-1) n^{2(HK-1)}$ for large
$n$. Let us study the behavior of the term $f_{a}(n)$ for large $n$.
We have
\begin{align*}
f_{a}(n)&= \left(  (a+1) ^{2H} + (a+n+1)^{2H} \right) ^{K} -\left(  (a+1) ^{2H} + (a+n)^{2H} \right) ^{K}\\
&\hskip 10mm- \left(  a ^{2H} + (a+n+1)^{2H} \right) ^{K}+ \left(  a^{2H} + (a+n)^{2H} \right) ^{K}\\
&= n^{2HK} \Bigl[ \left(   \left((a+1)n^{-1} \right) ^{2H}  + \left( (a+1)n^{-1} +1\right) ^{2H} \right) ^{K} \\
& \hskip 20mm -\left(
\left( (a+1)n^{-1} \right) ^{2H}  + \left( {a}{n}^{-1}+1 \right) ^{2H} \right) ^{K}\\
& \hskip 20mm- \left(   \left( {a}{n} ^{-1}\right) ^{2H}  + \left( (a+1){n}^{-1}+1\right) ^{2H} \right) ^{K}\\
& \hskip 20mm+ \left(   \left(
{a}{n}^{-1} \right) ^{2H}  + \left( {a}{n}^{-1}+1\right) ^{2H} \right) ^{K}\Bigr]\\
&= n^{2HK} \Bigg [ \bigg ( (a+1) ^{2H} n^{-2H} +1 \\
& \hskip 15mm
+ \sum_{j=0}^{\infty}((j+1)!)^{-1}2H(2H-1)\cdots (2H-j)(a+1) ^{j+1}
n^{-j-1}\bigg) ^{K} \\
&\hskip 12mm- \bigg( (a+1) ^{2H} n^{-2H} +1\\
&\hskip 15mm + \sum_{j=0}^{\infty}((j+1)!)^{-1}2H(ZH-1)\cdots (2H-j)a ^{j+1}
n^{-j-1}\bigg) ^{K}  \\
&\hskip 12mm-\bigg( a ^{2H} n^{-2H} +1\\
& \hskip 15mm  +
\sum_{j=0}^{\infty}((j+1)!)^{-1}2H(2H-1)\cdots (2H-j)(a+1) ^{j+1}
n^{-j-1}\bigg) ^{K} \\\
&\hskip 12mm+ \bigg( a ^{2H} n^{-2H} +1 \\
&
\hskip 15mm +\sum_{j=0}^{\infty}((j+1)!)^{-1} 2H(2H-1)\cdots (2H-j)^{j+1}a ^{j+1}
n^{-j-1}\bigg) ^{K} \Bigg]. \\
\end{align*}
By the asymptotic behavior of the function $(1+ y) ^{K}$ when $y\to0$ we obtain

\begin{align*}
f_{a}(n)
&=n^{2HK}\Bigg[  1 + \sum_{\ell =0}^{\infty}((l+1)!)^{-1}K(K-1)\cdots (K-\ell)\\
& \hskip 5mm \times \bigg( (a+1) ^{2H} n^{-2H} +1 \\
& \hskip 10mm +
\sum_{j=0}^{\infty}((j+1)!)^{-1}2H(2H-1)\cdots (2H-j)(a+1) ^{j+1}
n^{-j-1}\bigg)^{\ell +1} \\
&  \hskip 3mm- 1-\sum_{\ell =0}^{\infty}((l+1)!)^{-1}K(K-1)\cdots (K-\ell) \\
& \hskip 5mm\times \bigg( (a+1) ^{2H} n^{-2H} +1 \\
& \hskip 10mm +
\sum_{j=0}^{\infty}((j+1)!)^{-1}2H(2H-1)\cdots (2H-j)a ^{j+1}
n^{-j-1}\bigg)^{\ell +1} \\
&  \hskip 3mm- 1-\sum_{\ell =0}^{\infty}((l+1)!)^{-1}K(K-1)\cdots (K-\ell)\\
& \hskip 5mm\times \bigg( a^{2H} n^{-2H} +1\\
& \hskip 10mm +
\sum_{j=0}^{\infty}((j+1)!)^{-1}2H(2H-1)\cdots (2H-j)a ^{j+1}
n^{-j-1}\bigg)^{\ell +1} \\
&  \hskip 3mm +1+\sum_{\ell =0}^{\infty}((\ell +1)!)^{-1}K(K-1)\cdots (K-\ell)\\
& \hskip 5mm\times \bigg( a ^{2H} n^{-2H} +1 \\
& \hskip 10mm+
\sum_{j=0}^{\infty}((j+1)!)^{-1}2H(2H-1)\cdots (2H-j)a ^{j+1}
n^{-j-1}\bigg)^{\ell +1} \Bigg] \\
&= n^{2HK}\Bigg[  1 + \sum_{\ell =1}^{\infty}((l+1)!)^{-1}K(K-1)\cdots (K-\ell)\\
&\hskip 5mm\times \bigg( (a+1) ^{2H} n^{-2H} +1\\
& \hskip 10mm +
\sum_{j=0}^{\infty}((j+1)!)^{-1}2H(2H-1)\cdots (2H-j)(a+1) ^{j+1}
n^{-j-1}\bigg)^{\ell +1} \\
&  \hskip 3mm- 1-\sum_{\ell =1}^{\infty}((l+1)!)^{-1}K(K-1)\cdots (K-\ell)\\
&\hskip 5mm\times \bigg( (a+1) ^{2H} n^{-2H} +1 \\
& \hskip 10mm+
\sum_{j=0}^{\infty}((j+1)!)^{-1}2H(2H-1)\cdots (2H-j)a ^{j+1}
n^{-j-1}\bigg)^{\ell +1} \\
&  \hskip 3mm- 1-\sum_{\ell =1}^{\infty}((l+1)!)^{-1}K(K-1)\cdots (K-\ell)\\
&\hskip 5mm\times \bigg( a^{2H} n^{-2H} +1\\
& \hskip 10mm +
\sum_{j=0}^{\infty}((j+1)!)^{-1}2H(2H-1)\cdots (2H-j)a ^{j+1}
n^{-j-1}\bigg)^{\ell +1}  \\
&  \hskip 3mm+1+\sum_{\ell =1}^{\infty}((l+1)!)^{-1}K(K-1)\cdots (K-\ell)\\
&\hskip 5mm\times \bigg( a ^{2H} n^{-2H} +1 \\
& \hskip 10mm+
\sum_{j=0}^{\infty}((j+1)!)^{-1}2H(2H-1)\cdots (2H-j)a ^{j+1}
n^{-j-1}\bigg)^{\ell +1} \Bigg] \\
&=2^{-1}K(K-1) n^{2HK} \Big[ \left( (a+1) ^{2H} n^{-2H} +2H (a+1) n^{-1} \right.\\
& \left .\hskip 50mm + H (2H-1 ) (a+1) ^{2} n^{-2} \right)^{2} \\
& \hskip 10mm -
\left( (a+1) ^{2H} n^{-2H} +2H an^{-1} + H (2H-1 ) a ^{2} n^{-2} \right)^{2} \\
& \hskip 10mm-\left( a ^{2H} n^{-2H} +H (a+1) n^{-1} + 2H (2H-1 ) (a+1) ^{2} n^{-2} \right)^{2} \\
&\hskip 10mm+ \left( a^{2H} n^{-2H} +H a n^{-1} + 2H (2H-1 ) a ^{2} n^{-2} \right)^{2} \Big] \\
& \hskip 10mm+ \cdots \\
&=HK(K-1) \left( (a+1) ^{2H}-a^{2H}\right) n^{2(HK-1)+(1-2H)}+
\cdots .
\end{align*}
This completes the proof. \end{proof} \vskip 3mm

Let  us discuss  some consequences of the theorem above.

\begin{remark}
What is the main term in $R(a, a+n)$?
Note that   $H>\frac12$ \,\,if and only if\,\,
$2(HK-1)>2(HK-1)+(1-2H)$. Consequently the dominant term for
$R(a,a+n) $ is of order $n^{2HK-2}$ if $H>\frac{1}{2}$  and of order
$n^{2HK-1-2H}$ if $H<\frac{1}{2}$.
\end{remark}

\vskip 3mm

Another interesting observation is that, although the main term of
the covarinace function  $R(a,a+n)$ changes depending on
whether $H$ is bigger or less than one half, the long-range
dependence of the process $B^{H,K}$ depends on the value of the product $HK$.

\begin{corollary}
For integers $a\geq 1$ and $n\geq 0$, let $R(a,a+n)$ be given by
{\rm (\ref{rbi})}. Then for every $a\in\N$, we have
\begin{equation*}
\sum_{n\geq 0} R(a, a+n) =\infty \hskip0.5cm \mbox{ if } 2HK>1
\end{equation*}
and
\begin{equation*}
\sum_{n\geq 0} R(a, a+n) <\infty \hskip0.5cm \mbox{ if } 2HK\leq1 .
\end{equation*}

\end{corollary}
\begin{proof} Suppose first that $2HK>1$. Then it forces $H$ to
be bigger than $\frac{1}{2}$ and the dominant term of $R(a, a+n)$
is $n^{2HK-2}$ when $n$ is large.  So the series diverges.

Suppose that $2HK<1$. If $H>\frac{1}{2}$, the main term of $R(a,a+n)$ 
is $n^{2HK-2} $ and the series converges. If $H<\frac{1}{2}$,
then the main term is $n^{2HK-2H-1} $ and the series converges again.

If $2HK=1$ (and thus  $H>\frac{1}{2}$) then $R(a, a+n)$ behaves as
$n\to \infty $ as $n^{-1-2H}$ and the series is convergent.
\end{proof}

\begin{corollary}Let $R^{X^{H,K}}(a,a+n)$ be the noise defuned in {\rm (\ref{cor})}. Then
\label{thm:XHK}
\begin{align*}
R&^{X^{H,K}}(a, a+n) \\
& = {\Gm(1-K)}{K}^{-1}\left (-4HK(K-1)\left((a+1)^{2H}-a^{2H}\right) n^{2(HK-1)+(1-2H)} + \cdots\right ).
\end{align*}
\end{corollary}
\begin{proof} It follows from Theorem 4.1  and the fact that the
covariance function of the fractional Brownian motion with
Hurst parameter $HK$ behaves as $HK(2HK-1)n^{2(HK-1)}$ when $n\to\infty$.
\end{proof}

\vskip 10mm


\section{More on the process $X^{H,K}$}
We will give few additional properties of the process $X^{K}$
defined in (1.1). Recall (1.2) that for every $s,t\geq 0$
\begin{equation*}
R^{X^{K}}(s,t):=E[X^{K}_{s}X^{K}_{t}] =\Gm(1-K)K^{-1}(t^{K} + s^{K}
-(t+s) ^{K}).
\end{equation*}
Denote by $B^{K/2}=(B^{K/2}_{t}, {t\in \mathbb{R}})$ a fractional
Brownian motion with Hurst index $K/2$ defined for all $t\in\R$ and
let
\begin{equation*}
B^{o, K/2}_{t}= {2}^{-1} \left( B^{K/2}_{t} -B^{K/2}_{-t}\right),
\hskip0.5cm B^{e, K/2}_{t}= {2}^{-1} \left(B^{K/2}_{t}
+B^{K/2}_{-t}\right).
\end{equation*}
The processes $B^{o, K/2} $ and $B^{e, K/2}$ are respectively the
odd part and the even part of the fractional Brownian motion
$B^{K/2}$. Denote by $R^{o, K/2} $ the covariance of the process
$B^{o, K/2}_{t}$, by $R^{e, K/2}$ the covariance of the $B^{e, K/2}$
and by $R^{B^{K/2}}$ the covariance of the fractional Brownian
motion $B^{K/2}$. We have the following facts:
\begin{equation*} \label{XB}
R^{X^{K}}(t,s)= C_{3}R^{B^{K/2}} (t, -s)=C_{3} R^{B^{K/2}} (s, -t)
\end{equation*}
where $C_{3}= 2\Gamma (1-K) K^{-1}$, and
\begin{equation*} \label{rx1}
R^{o, K/2}(t,s)=\frac{1}{2} \left( R^{B^{K/2}}(t,s)- R^{B^{K/2}}
(t,-s)\right)
\end{equation*}
and
\begin{equation*} \label{rx2}
R^{e, K/2}(t,s)=\frac{1}{2} \left( R^{B^{K/2}}(t,s)+ R^{B^{K/2}}
(t,-s)\right).
\end{equation*}
As a consequence
\begin{equation*}
R^{e, K/2}(t,s)-R^{o, K/2}(t,s)=R^{B^{K/2}}(t, -s)=
C_{3}^{-1}R^{X^{K}}(t,s) .
\end{equation*}

From the above computations, we obtain
\begin{proposition}
We have the following equality
$$
C_{3}^{-{1}/{2}}X^{K} + B^{e, K/2} \eqd B^{o, K/2}
$$
if $X^{K}$ and $B^{e, K/2}$ are independent.
\end{proposition}

Let us  go back to the bifractional Brownian noise $R(a,a+n)$ given in (\ref{rbi}). 
By the decomposition (\ref{deco}), we have
$$
C_{1} X^{H,K} + B^{H,K} \eqd C_{2} B^{HK},
$$
where $C_1$ and $C_2$ are as before, and thus we get
\begin{align*}
R(a, a+n)&= C_{2}^{2} R^{B^{HK}}(a, a+n) -C_{1}^{2}R^{X^{H,K}} (a, a+n) \\
&= C_{2}^{2} R^{B^{HK}}(0,n) -C_{3}\left( R^{e, K/2, H}(a,a+n) -
R^{o,K/2, H}(a,a+n)   \right)
\end{align*}
where $R^{e, K/2, H}(a,a+n)$ is the noise of the process $B^{e, K/2}
_{t^{2H}}, t\geq 0$.
\begin{remark}
The fact that the covariance function $ R^{X^{K}}(a, a+n)$ of the
process $X^{K/2}$ converges to zero as $a \to \infty $ can be
interpreted as \lq\lq the covariance of the odd part"
$C_{3}R^{B^{o, K/2}}(a,a+n)$ and \lq\lq the covariance of the even
part" $C_{3}R^{B^{e,K/2}}(a,a+n)$ have the same limit
$2^{-1}C_{1}^{2}R^{B^{K/2}}(0,n)$ when $a$ tends to infinity.
\end{remark}
\vskip 10mm

\section{Limit theorems to the bifractional Brownian motion}
In this section, we prove two limit theorems to the bifractional
Brownian motion. Throughout this section, we use the following
notation. Let $0<H<1, 0<K<1$ such that $2HK>1$ and let $(\xi_j,
j=1,2,...)$ be a sequence of standard normal random variables.
Define a function $g(x,y), x\ge 0, y\ge 0$ by
\begin{align}
g(x,y) & = 2^{2-K}H^2K(K-1)(x^{2H}+y^{2H})^{K-2}(xy)^{2H-1}\nonumber\\
& \hskip 20mm + 2^{1-K}HK(2HK-1)|x-y|^{2HK-2}\nonumber\\
& =: g_{1}(x,y) + g_{2}(x,y),\label{g}
\end{align}
for $(x,y)$ with $x\ne y$ and $x\ne 0$ and $y\ne 0$.
\newcommand{\sumjt}{\sum_{j=1}^{[nt]}}
\newcommand{\sumit}{\sum_{i=1}^{[nt]}}
\newcommand{\sumjs}{\sum_{j=1}^{[ns]}}
\begin{proposition}
\label{prop:conv} Under the notation above, assume that
$E[\xi_i\xi_j]= g(i,j)$. Then
$$
\left (n^{-HK}\sumjt \xi_j, t\ge 0\right ) \eqc \left (B_t^{H,K},
t\ge 0\right ).
$$

\end{proposition}
To prove this, we need a lemma.
\begin{lemma}
$$
\int_0^t\int_0^s g(u,v)dudv = 2^{-K}\left [ (t^{2H}+
s^{2H})^K-|t-s|^{2HK}\right ].
$$
\end{lemma}
\begin{proof} It follows easily from the fact that
$\frac{\partial^{2}R} {\partial x\partial y} (x,y)= g(x,y)$ for
every $x,y\geq 0$ and by using that $2HK>1$.

\end{proof}

\begin{proof}({\it Proof of Proposition~{\rm \ref{prop:conv}}.}) It is enough to
show that
\begin{align*}
I_n &:=E\left [\left (n^{-HK}\sumit \xi_i\right) \left ( n^{-HK}\sumjs\xi_j\right )\right ]\\
& \to E[B_t^{H,K}B_s^{H,K}] = 2^{-K}\left ( (t^{2H}+s^{2H})^K-
|t-s|^{2HK}\right).
\end{align*}
We have
$$
I_n = n^{-2HK}\sumit\sumjs E[\xi_i\xi_j] = n^{-2HK}\sumit\sumjs
g(i,j).
$$
Observe that
\begin{align}
g\left ( \frac{i}{n}, \frac{j}{n}\right ) & = 2^{2-K}H^2K(K-1)\left
(\left (\frac{i}{n}\right )^{2H} +\left (\frac{j}{n}\right )^{2H}
\right )^{K-2}
\left (\frac{ij}{n^2}\right )^{2H-1}\nonumber \\
& \hskip 20mm  +2^{1-K}HK(2HK-1)\left |\frac{i}{n}- \frac{j}{n}\right |^{2HK-2}\nonumber \\
& = 2^{2-K}H^2K(K-1)n^{-2H(K-2)-2(2H-1)}(i^{2H}+j^{2H})^{K-2}(ij)^{2H-1}\nonumber \\
& \hskip 20mm + 2^{1-K}HK(2HK-1)n^{-2HK+2}|i-j|^{2HK-2}\nonumber \\
& = n^{2(1-HK)} \Bigl ( 2^{2-K}H^2K(K-1)(i^{2H}+j^{2H})^{K-2}(ij)^{2H-1}\nonumber \\
& \hskip 30mm + 2^{1-K}HK(2HK-1)|i-j|^{2HK-2}\Bigr )\nonumber \\
& = n^{2(1-HK)}g(i,j). \label{gij}
\end{align}
Thus, as $n\to\infty$,
\begin{align*}
I_n & = n^{-2HK}\sumit\sumjs n^{2HK-2}g\left
(\frac{i}{n},\frac{j}{n}\right )
= n^{-2} \sumit\sumjs g\left (\frac{i}{n},\frac{j}{n}\right ) \\
& \to \int_0^t\int_0^s g(u,v)dudv
= 2^{-K}\left ((t^{2H}+s^{2H})^K - |t-s|^{2HK}\right )\\
&= E[B_t^{H,K}B_s^{H,K}].
\end{align*}
\end{proof} \vskip 3mm

\begin{remark}
This result seems easy to be  generalized to more general Gaussian
selfsimilar processes such that their covariance $R$ satisfies
$\frac{\partial R}{\partial x \partial y} \in L^{2} \left(  (0,
\infty ) ^{2} \right)$.
\end{remark}

We next consider more general sequence of nonlinear functional of
standard normal random variables. Let $f$ be a real valued
function such that $f(x)$ does not vanish on a set of positive
measure, $E[f(\xi_1)]=0$ and $E[(f(\xi_1))^2]<\infty$. Let $H_k(x)$
denotes the $k$-th Hermite polynomial with highest coefficient 1. We
expand $f$ as follows (see e.g. \cite{DM}):
$$
f(x) = \sum_{k=1}^{\infty}c_kH_k(x),
$$
where $\sum_{k=1}^{\infty}c_k^2k!<\infty,
c_k=E[f(\xi_i)H_k(\xi_j)]$. This expansion is possible under the
assumption $Ef(\xi_1)=0$ and $E[(f(\xi_1))^2]<\infty$. Assume that
$c_1\ne 0$. Now consider a new sequence
$$
\eta_j = f(\xi_j), j=1,2,...,
$$
where $(\xi_j , j=1,2,...)$ is the same sequence of standard normal
random variables as before.
\begin{proposition}
\label{prop:gene-conv} Under the same assumptions of
Proposition~\ref{prop:conv}, we have
$$
\left (n^{-HK}\sumjt \eta_j, t\ge 0\right ) \eqc \left
(c_1B_t^{H,K}, t\ge 0\right ).
$$
\end{proposition}
\begin{proof} Note that $\eta_j = f(\xi_j) = c_1\xi_j +
\sum_{k=2}^{\infty}c_kH_k(\xi_j)$. We have
$$
n^{-HK}\sum_{j=1}^{[nt]}\eta_j = c_1 n^{-HK}\sumjt\xi_j +
n^{-HK}\sumjt\sum_{k=2}^{\infty}c_kH_k(\xi_j).
$$
By Proposition~\ref{prop:conv}, it is enough to show that
$$
E\left [\left ( n^{-HK}\sumjt\sum_{k=2}^{\infty}c_kH_k(\xi_j)\right
)^2\right ]\to 0 \quad\text{as}\,\, n\to\infty.
$$
We have
\begin{align*}
J_n&:= E\left [\left ( n^{-HK}\sumjt\sum_{k=2}^{\infty}c_kH_k(\xi_j)\right )^2\right ]\\
& =
n^{-2HK}\sumit\sumjt\sum_{k=2}^{\infty}\sum_{\ell=2}^{\infty}c_kc_{\ell}E[H_k(\xi_j)H_{\ell}(\xi_j)].
\end{align*}
In general, if $\xi$ and $\eta$ are jointly Gaussian random
variables with $E[\xi]=E[\eta]=0$, $E[\xi^2]=E[\eta^2]=1$ and
$E[\xi\eta]=r$, then
$$
E[H_k(\xi)H_{\ell}(\eta)] = \dl_{k,\ell}r^kk!,
$$
where
$$
\dl_{k,\ell}=
\begin{cases}
1, & k=\ell,\\
0, & k\ne\ell .
\end{cases}
$$
Thus
\begin{align*}
J_n&= n^{-2HK}{[nt]} \sum_{\ell=2}^{\infty } c_{\ell}^{2}\ell!+
 n^{-2HK}\sum _{i,j=1; i\not= j}^{[nt]}\sum_{\ell=2}^{\infty}c_{\ell}^2 (E[\xi_i\xi_j])^{\ell}{\ell}!\\
& =n^{-2HK}{[nt]} \sum_{\ell=2}^{\infty }
c_{\ell}^{2}\ell!+n^{-2HK}\sum _{i,j=1; i\not= j}^{[nt]}\sum_{\ell
=2}^{\infty}c_{\ell}^2g(i,j)^{\ell}\ell !
\end{align*}
Since for every $i,j\geq 1$ ($i\not=j$) one has $\vert g(i,j)\vert
\leq \left( E[\xi_{i} ^{2} ]\right) ^{{1}/{2} } \left( E[\xi_{j}
^{2}] \right) ^{{1}/{2} }= 1$, we get
\begin{align*}
J_{n} & \leq n^{-2HK}{[nt]} \sum_{\ell=2}^{\infty }
c_{\ell}^{2}\ell! +n^{-2HK}\sum_{\ell =2}^{\infty}c_{\ell }^{2}\ell!
\sum_{i,j=1; i\not=j } ^{[nt]}  g(i,j)^{2} \\
&\leq t n^{1-2HK} \sum_{\ell =2} ^{\infty } c_{\ell}^{2} \ell! +
n^{2(HK-1)} \left (\sum_{\ell=2}^{\infty}c_{\ell}^2{\ell}!\right )
\left( n^{-2}\sum_{i,j=1; i\not= j} ^{[nt]}  g\left( \frac{i}{n},
\frac{j}{n}\right)^{2}\right) ,
\end{align*}
where we have used (\ref{gij})

Here as $n\to \infty$, since
 $\sum _{\ell=2}^{\infty } c_{\ell
}^{2} \ell ! <\infty $  and $2HK>1$ we obtain that \\ $t n^{1-2HK}
\sum_{\ell =2} ^{\infty } c_{\ell}^{2} \ell!$ converges to zero as
$n\to \infty$. On the other hand, with $C$ an absolute positive
constant and $g_{1}$ and $g_{2}$   given by (\ref{g}),
$$
n^{-2}\sum_{i,j=1; i\not= j} ^{[nt]}  g\left(
\frac{i}{n},\frac{j}{n}\right)^{2}\leq C n^{-2}\left(\sum_{i,j=1;
i\not= j} ^{[nt]} g_{1} \left( \frac{i}{n},\frac{j}{n}\right)^{2}+
\sum_{i,j=1; i\not= j} ^{[nt]} g_{2}\left(
\frac{i}{n},\frac{j}{n}\right)^{2}\right).
$$
The first sum $n^{-2}\sum_{i,j=1; i\not= j} ^{[nt]} g_{1} \left(
\frac{i}{n},\frac{j}{n}\right)^{2}$   is a Riemann sum converging to
the integral $\int_{0}^{t}\int_{0}^{t}g_{1}^{2}(x,y) dxdy $.  Note
that this integral is finite because $\vert g_{1}(x,y)\vert \leq C
(xy) ^{HK-1}$  and the integral $\int_{0}^{t}\int_{0}^{t} \vert
x-y\vert ^{2HK-2} dxdy$ is finite when $2HK>1$.  Since $ n^{2(HK-1)}\to
0$ we easily get $$n^{2(HK-1)} n^{-2}\sum_{i,j=1; i\not= j} ^{[nt]}
g_{1} \left( \frac{i}{n},\frac{j}{n}\right)^{2}\to 0$$ as $n\to
\infty$.

The second sum involving $g_{2}$ appears in the classical fractional Brownian case
because it is, modulo a constant, the second derivative of the
covariance of the fractional Brownian motion with Hurst parameter
$HK$.  The convergence of $$n^{2(HK-1)}n^{-2}\sum_{i,j=1; i\not= j}
^{[nt]} g_{2}\left( \frac{i}{n},\frac{j}{n}\right)^{2}$$ has been
already proved in e.g. \cite{DM}. The proof is completed.\end{proof}

\par\bigskip\noindent
{\bf Acknowledgment.} 
The authors are grateful to a referee for his/her useful comments for making
the final version of the paper.

\bibliographystyle{amsplain}

\end{document}